\documentclass[12pt,letterpaper]{amsart}
\usepackage{geometry}
\usepackage{hyperref}
\usepackage{mathpazo}
\usepackage{amssymb,url,setspace,xcolor}
\newtheorem{theorem}{Theorem}
\newtheorem*{theorem*}{Theorem}
\newtheorem{proposition}{Proposition}
\newtheorem{lemma}{Lemma}

\theoremstyle{remark}

\onehalfspace

\title[Hankel Operators on Bergman spaces of Reinhardt Domains]{Hankel Operators on the Bergman spaces of Reinhardt
Domains and Foliations of Analytic Disks}

\author{Timothy G. Clos}
\address[Timothy G. Clos]{Bowling Green State University, 
	Department of Mathematics and Statistics,  Bowling Green, Ohio 43403 }
\email{clost@bgsu.edu}

\subjclass[2010]{Primary 47B35; Secondary 32W05}
\keywords{Hankel operators, Reinhardt domains, analytic disks}
\date{\today}
\begin{document}
\begin{abstract}
Let $\Omega\subset \mathbb{C}^2$ be a bounded pseudoconvex complete Reinhardt domain with a smooth boundary.
We study the behavior of analytic structure in the boundary of $\Omega$ and obtain a compactness result for Hankel operators on 
the Bergman space of $\Omega$.

\end{abstract}
\maketitle

\section{Introduction} 
Let $\Omega\subset \mathbb{C}^n$ for $n\geq 2$ be a bounded domain.  We let $dV$ be the (normalized) Lebesgue volume measure on $\Omega$.  Then $L^2(\Omega)$ is the space of measurable, square integrable functions on $\Omega$.  Let $\mathcal{O}_{\Omega}$ be the collection of all holomorphic (analytic) functions on $\Omega$.  Then the Bergman space $A^2(\Omega):=\mathcal{O}_{\Omega}\cap L^2(\Omega)$ is a closed subspace of $L^2(\Omega)$, a Hilbert space.  Therefore, there exists an orthogonal projection
$P:L^2(\Omega)\rightarrow A^2(\Omega)$ called the Bergman projection.  Then the Hankel operator with symbol $\phi\in L^{\infty}(\Omega)$ is defined as \[H_{\phi}f:=(I-P)(\phi f)\] where $I$ is the identity operator and $f\in A^2(\Omega)$.

\section{Previous Work}
  
   Compactness of Hankel operators on the Bergman spaces of bounded domains and its relationship between analytic structure in the boundary of these domains is an ongoing research topic.  In one complex dimension, Axler in \cite{Axler86} completely characterizes compactness of Hankel operators with conjugate holomorphic, $L^2$ symbols.  There, the emphasis is on whether the symbol belongs to the little Bloch space.  This requires that the derivative of the complex conjugate of the symbol satisfy a growth condition near the boundary of the domain. \\

 The situation is different in several variables for conjugate holomorphic symbols.  In \cite{clos}, the author completely characterizes compactness of Hankel operator with conjugate holomorphic symbols on convex Reinhardt domains in $\mathbb{C}^n$ if the boundary contains a certain class of analytic disks.  The proof relied on using the analytic structure in the boundary to show that a compact Hankel operator with a conjugate holomorphic symbol must be the zero operator, assuming certain conditions on the boundary of the domain.  In particular, the symbol is identically constant if certain conditions are satisfied.  An example of a domain where these conditions are satisfied is the polydisk in $\mathbb{C}^n$ (as seen in \cite{Le10} and \cite{clos}).\\
 
 In \cite{CelZey} the authors studied the compactness of Hankel operators with symbols continuous up to the closure of bounded pseudoconvex domains via compactness multipliers.  They showed if $\phi\in C(\overline{\Omega})$ is a compactness multiplier then $H_{\phi}$ is compact on $A^2(\Omega)$.  The authors of \cite{CelZey} approached the problem using the compactness estimate machinery developed in \cite{StraubeBook}.\\  
 
Hankel operators with symbols continuous up to the closure of the domain is also studied in \cite{CuckovicSahutoglu09} and \cite{ClosSahut}.    The paper \cite{CuckovicSahutoglu09} considered Hankel operators with symbols that are $C^1$-smooth up to the closure of bounded convex domains in $\mathbb{C}^2$.  The paper \cite{ClosSahut} considered symbols that are continuous up to the closure of bounded convex Reinhardt domains in $\mathbb{C}^2$.  Thus the regularity of the symbol was reduced at the expense of a smaller class of domains.\\
   
   Many of these results characterize the compactness of these operators by the behavior of the symbol along analytic structure in the domain.  For bounded pseudoconvex domains in $\mathbb{C}^n$, compactness of the $\overline{\partial}$-Neumann operator implies the compactness of Hankel operators with symbols continuous up to the closure of the domain.  See \cite{FuSt} and \cite{StraubeBook} for more information on compactness of the $\overline{\partial}$-Neumann operator.  For example the ball in $\mathbb{C}^n$ has compact $\overline{\partial}$-Neumann operator and hence any Hankel operator with symbol continuous up the closure of the ball is compact on the Bergman space of the ball.  The compactness of the $\overline{\partial}$-Neumann operator on the ball in $\mathbb{C}^n$ follows from the convexity of the domain and absence of analytic structure in the boundary of the domain.  See \cite{StraubeBook}.\\
   
 As shown in \cite{dbaressential}, the existence of analytic structure in the boundary of bounded convex domains is an impediment to the compactness of the $\overline{\partial}$-Neumann operator.  It is therefore natural to ask whether the Hankel operator with symbol continuous up to the closure of the domain can be compact if the $\overline{\partial} $-Neumann operator is not compact.  As we shall see, the answer is yes.  On the polydisk in $\mathbb{C}^n$, \cite{Le10} showed that the answer to this question is yes, despite the non-compactness of the $\overline{\partial}$-Neumann operator.  For bounded convex domains in $\mathbb{C}^n$ for $n\geq 2$, relating the compactness of the Hankel operator with continuously differentiable symbols to the geometry of the boundary is well studied.  See \cite{CuckovicSahutoglu09}.  They give a more general characterization than \cite{Le10} for symbols that are $C^1$-smooth up to the closure of the domain.  For symbols that are only continuous up to the closure of bounded convex Reinhardt domains in $\mathbb{C}^2$, there is a complete characterization in \cite{ClosSahut}. \\     
    
 \section{The Main Result} 
 
 In this paper we investigate the compactness of Hankel operators on the Bergman spaces of smooth bounded pseudoconvex complete Reinhardt domains.  These domains may not be convex as in \cite{ClosSahut} but are instead almost locally convexifiable.  That is, for any $(p_1,p_2)\in b\Omega$ and if $(p_1,p_2)$ are away from the coordinate axes, then there exists $r>0$ so that 
\[B((p_1,p_2),r):=\{(z_1,z_2)\in \mathbb{C}^2: |z_1-p_1|^2+|z_2-p_2|^2<r^2\} \] and a biholomorphism $T:B((p_1,p_2),r)\rightarrow \mathbb{C}^2$ so that $ B((p_1,p_2),r)\cap \Omega$ is a domain and $T(B((p_1,p_2),r)\cap \Omega)$ is convex.  We will use this fact along with a result in \cite{CuckovicSahutoglu09} to localize the problem.  We then analyze the geometry on analytic structure in the resulting convex domain.  Then we perform the analysis on the boundary of this convex domain using the boundary geometry previously established to show the main result. \\

 \begin{theorem}\label{thmmain}
 Let $\Omega\subset\mathbb{C}^2$ be a bounded pseudoconvex complete Reinhardt domain with a smooth boundary.  Then $\phi\in C(\overline{\Omega})$ so that $\phi\circ f$ is holomorphic for any holomorphic $f:\mathbb{D}\rightarrow b\Omega$ if and only if $H_{\phi}$ is compact on $A^2(\Omega)$.
\end{theorem}

We will assume $\phi\circ f$ is holomorphic for any holomorphic function $f:\mathbb{D}\rightarrow b\Omega$ and show that $H_{\phi}$ is compact on $A^2(\Omega)$, as the converse of this statement appears as a corollary in \cite{CCS}.

\section{Analytic structure in the boundary of pseudoconvex complete Reinhardt domains in $\mathbb{C}^2$}

We fill first investigate the geometry of non-degenerate analytic disks in the boundary of Reinhardt domains.  We define the following collection for any bounded domain $\Omega\subset \mathbb{C}^n$.
\[\Gamma_{\Omega}:=\overline{\bigcup_{f\in A(\mathbb{D})\cap C(\overline{\mathbb{D}})\, , f \,\text{non-constant}}\{f(\mathbb{D}) |f:\mathbb{D}\rightarrow b\Omega\}}\]

Let $\Omega\subset \mathbb{C}^n$ for $n\geq 2$ be a domain.  
We say $\Gamma\subset b\Omega$ is an analytic disk if there exists $F:\mathbb{D}\rightarrow \mathbb{C}^n$ so that every component function of $F$ is holomorphic on $\mathbb{D}$ and continuous up to the boundary of $\mathbb{D}$ and $F(\mathbb{D})=\Gamma$. \\

One observation is for any Reinhardt domain $\Omega\subset \mathbb{C}^n$, if $F(\mathbb{D})\subset b\Omega$ is an analytic disk where $F(\zeta):=(F_1(\zeta), F_2(\zeta),...,F_n(\zeta))$, then for any $(\theta_1,\theta_2,...,\theta_n)\in \mathbb{R}^n$, $G(\mathbb{D})\subset b\Omega$ is also an analytic disk where
\[G(\zeta):=(e^{i\theta_1}F_1(\zeta), e^{i\theta_2}F_2(\zeta),...,e^{i\theta_n}F_n(\zeta)).\]

We say an analytic disk $f(\mathbb{D})$ where $f=(f_1,f_2,...,f_n)$ is trivial or degenerate if $f_j$ is identically constant for all $j\in \{1,2,...,n\}$.  Otherwise, we say an analytic disk is non-trivial or non-degenerate. \\   

Let $\Omega\subset \mathbb{C}^2$ be a bounded pseudoconvex complete Reinhardt domain with a smooth boundary.  If $g(\mathbb{D})\subset b\Omega$ is an analytic disk so that 
$\overline{g(\mathbb{D})}\cap \{z_2=0\}\neq \emptyset$ or $\overline{g(\mathbb{D})}\cap \{z_1=0\}\neq \emptyset$, then $g(\zeta)=(g_1(\zeta),0)$ or $g(\zeta)=(0,g_2(\zeta))$, respectively.

They are possibly infinitely many continuous families of non-trivial analytic disks in the boundary of bounded complete Reinhardt domains $\Omega$ in $\mathbb{C}^2$.  Hence by compactness of the boundary of $\Omega$, there are subsets of $b\Omega$ that are accumulation sets of families of analytic disks.  This next lemma gives us some insight on the structure of these accumulation sets.

\begin{lemma}\label{disklim}
Suppose $\Omega\subset \mathbb{C}^2$ is a bounded complete Reinhardt domain and $\{\Gamma_j\}_{j\in \mathbb{N}}\subset b\Omega$ is a sequence of pairwise disjoint, continuous families of analytic disks so that $\Gamma_j\rightarrow \Gamma_0$ as $j\rightarrow \infty$, where $\Gamma_0=\{e^{i\theta}F(\mathbb{D}):\theta\in [0,2\pi]\}$.  Then, there exists $c_1,c_2\in \mathbb{C}$ so that $F\equiv (c_1,c_2)$.

\end{lemma} 

\begin{proof}
Let $\sigma$ be the Lebesgue measure on the boundary.   Without loss of generality, we may assume $\Gamma_j$ are families of non-degenerate analytic disks and so we may assume $\sigma(\Gamma_j)>0$ for all $j\in \mathbb{N}$.  If $\sigma(\Gamma_0)>0$, then we consider the sequence of indicator functions on $\Gamma_j$, called $\chi_{\Gamma_j}$.  By assumption, $\chi_{\Gamma_j}\rightarrow \chi_{\Gamma_0}$ pointwise as $j\rightarrow \infty$.  Hence an application of Lebesgue dominated convergence theorem shows that $\sigma(\Gamma_j)\rightarrow \sigma(\Gamma_0)$, and so $\sigma(\Gamma_j)\geq \delta>0$ for sufficiently large $j\in \mathbb{N}$.  Since $\Gamma_j$ are pairwise disjoint and $\Omega$ is bounded, this is a contradiction.  So $\sigma(\Gamma_0)=0$.  Now assume $\Lambda_j(\zeta):=(f_j(\zeta),g_j(\zeta))$ where $f_j$, $g_j$ are holomorphic on $\mathbb{D}$ and continuous up to the boundary of $\mathbb{D}$.  Furthermore, \[\Gamma_j=\{e^{i\theta}\Lambda_j:\theta\in [0,2\pi]\}.\]  Then, there exists $f,g$ so that 
\[\sup\{\text{dist}(((f_j(\zeta),g_j(\zeta)), f(\zeta),g(\zeta))):\zeta\in \overline{\mathbb{D}}\}\rightarrow 0\] as $j\rightarrow \infty$.  Therefore, one can show $f_j\rightarrow f$ and $g_j\rightarrow g$ uniformly on $\overline{\mathbb{D}}$ as $j\rightarrow \infty$.  So $f$ and $g$ are holomorphic on $\mathbb{D}$ and continuous on $\overline{\mathbb{D}}$.  To show $f$ and $g$ are constant it suffices to show they are constant on some open subset of $\mathbb{D}$.  Assume $f$ is not identically constant.  If $g$ is constant then by open mapping theorem, $F(\mathbb{D})$ is open in $\mathbb{C}\times \mathbb{R}$ and also $\sigma(F(\mathbb{D}))=0$, which cannot occur by the open mapping theorem.  So, we assume both $F$ and $g$ are not identically constant, so the zeros of $f'$ and $g'$ have no accumulation point in $\mathbb{D}$.  Then by a holomorphic change of coordinates, there exists an open simply connected set $D\subset \mathbb{D}$ so that $F(D)$ is biholomorphic to a subset $K$ of $\{z_1\in \mathbb{C}\}\times\{0\}$.  Hence again by the open mapping theorem, $f$ is constant on $K$ since $K$ has measure zero and so $f$ is constant on $\overline{\mathbb{D}}$ by the identity principle.

\end{proof}

\begin{lemma}\label{lembiholo}
Let $\Omega\subset \mathbb{C}^2$ be a bounded pseudoconvex complete Reinhardt domain with a 
smooth boundary.  Suppose $f:\mathbb{D}\rightarrow b\Omega$ and $g:\mathbb{D}\rightarrow b\Omega$ are holomorphic functions on $\mathbb{D}$ and continuous on $\overline{\mathbb{D}}$.
Assume that $\overline{f(\mathbb{D})}\cap \overline{g(\mathbb{D})}\neq \emptyset$.  Furthermore, assume $\overline{f(\mathbb{D})}\cap(\{z_1=0\}\cup \{z_2=0\})=\emptyset$ and $\overline{g(\mathbb{D})}\cap(\{z_1=0\}\cup \{z_2=0\})=\emptyset$.  Then, $f(\mathbb{D})$ and $g(\mathbb{D})$ are biholomorphically equivalent to analytic disks contained in a unique complex line.

\end{lemma}  

 \begin{proof}
 Let $\zeta_0\in \mathbb{D}$ and $\zeta_1\in \mathbb{D}$ be such that 
 $f(\zeta_0)=g(\zeta_1)$.  Without loss of generality, by composing with a biholomorphism of the unit disk that sends $\zeta_0$ to $\zeta_1$, we may assume $f(\zeta_0)=g(\zeta_0)$.  Then, there exists $r>0$ and a biholomorphism $T:B(f(\zeta_0),r)\rightarrow \mathbb{C}^2$ so that $f^{-1}(B(f(\zeta_0),r)\cap f(\mathbb{D}))\subset \mathbb{D}$ and $g^{-1}(B(f(\zeta_0),r)\cap g(\mathbb{D}))\subset \mathbb{D}$ and $T(B(f(\zeta_0),r)\cap \Omega)$ is convex.  Then, $A:=f^{-1}(B(f(\zeta_0),r)\cap f(\mathbb{D}))\cap g^{-1}(B(f(\zeta_0),r)\cap g(\mathbb{D}))$ is an open, non-empty, simply connected, and bounded.  By the Riemann mapping theorem, there exists a biholomorphism $R:\mathbb{D}\rightarrow A$.  Then, $T\circ f\circ R$ and $T\circ g\circ R$ are analytic disks in the boundary of a bounded convex domain.  Hence they are contained in a complex line by \cite[Lemma 2]{CuckovicSahutoglu09}.  In fact, they are contained in the same complex line because both disks have closures with non-empty intersection and the domain has a smooth boundary.  That is, if $L_{\alpha}:=\{(a_1\zeta+b_{\alpha},c_1\zeta+d_{\alpha}):\zeta\in \mathbb{C}\}$ and 
$L_{\beta}:=\{(a_2\zeta+b_{\beta},c_2\zeta+d_{\beta}):\zeta\in \mathbb{C}\}$ are one parameter continuous (continuously depending on the parameter) families of complex lines depending on parameters $\alpha$ and $\beta$ that locally foliate the boundary, with $(L_{\alpha_0}\cap L_{\beta_0})\cap b\Omega\neq \emptyset$, then $a_1=a_2$.  The argument uses the fact that boundary normal vectors must vary smoothly.  Furthermore, one can conclude $L_{\alpha_0}=L_{\beta_0}$ since one can show $b_{\alpha_0}=b_{\beta_0}$ and $d_{\alpha_0}=d_{\beta_0}$.
        
\end{proof}

\begin{proposition}\label{propconvex}
Let $\Omega\subset \mathbb{C}^2$ be a smooth bounded convex domain.  Let $\{\Gamma_j\}_{j\in \mathbb{N}}$ be a collection of analytic disks in $b\Omega$ so that \[\nabla:=\overline{\bigcup_{j\in \mathbb{N}}\Gamma_j}\] is connected.  Then there exists a convex set $S$ and a non-constant holomorphic function $F:\mathbb{D}\rightarrow b\Omega$ so that $F$ is continuous up to $\overline{\mathbb{D}}$, $F(\mathbb{D})=S$ and $\nabla\subset \overline{S}$.  
\end{proposition}

\begin{proof}
By Lemma \ref{lembiholo}, there exists a complex line $L=\mathbb{C}\times \{0\}$ so that $\nabla\subset L$ and by convexity of the domain, $L\cap \Omega=\emptyset$.  Then the convex hull of $\nabla$, called $\mathcal{H}(\nabla)$, is contained in $L\cap\overline{\Omega}$.  Since $\nabla$ contains a non-trivial analytic disk, the interior of $\mathcal{H}(\nabla)$ is non-empty.  We denote this non-empty interior as $I$.  Assume $\overline{I}\neq \mathcal{H}(\nabla)$.  Let $z_0\in \mathcal{H}(\nabla)\setminus \overline{I}$.  Then there is a positive Euclidean distance from $z_0 $ to $\overline{I}$.  Let $\mathcal{L}$ denote the collection of all line segments from $z_0$ to $bI$, called $K$.  Then $K$ has non-empty interior, which contradicts the convexity of $\mathcal{H}(\nabla)$.  Therefore, $I$ is a non-empty simply connected bounded open set in $\mathbb{C}$, so there is biholomorphism from $\mathbb{D}$ to $I$ that extends continuously to $\overline{\mathbb{D}}$ by smoothness of the boundary of $\Omega$.

\end{proof}

Then Lemma \ref{lembiholo} implies that any disk in the boundary of a bounded pseudoconvex complete Reinhardt domain in $\Omega\subset \mathbb{C}^2$ is contained in a continuous family of analytic disks, called $\Gamma$.  Furthermore, this continuous family can be represented as \[\Gamma=\{(e^{i\theta}F_1(\zeta), e^{i\theta}F_2(\zeta)):\theta\in [0,2\pi]\land \zeta\in \mathbb{D}\}\] since $b\Omega$ is three (real) dimensional and $\Gamma$ locally foliates $b\Omega$.

\section{Locally Convexifiable Reinhardt domains in $\mathbb{C}^2$}

\begin{lemma}\label{almost}
Let $\Omega\subset\mathbb{C}^n$ be a bounded pseudoconvex complete Reinhardt domain.  Then,
$\Omega$ is almost locally convexifiable.  That is, for every $(p_1,p_2,...,p_n)\in b\Omega\setminus(\{z_1=0\}\cup \{z_2=0\}\cup...\cup\{z_n=0\})$ there exists $r>0$ and there exists a biholomorphism $L$ on $B((p_1,p_2,...,p_n),r)$ so that $L(B((p_1,p_2,...,p_n),r)\cap \Omega)$ is convex.

\end{lemma}

Our understanding of analytic structure in the boundary of bounded convex domains is a crucial part of the proof the Theorem \ref{thmmain}.  The following proposition is proven in \cite{CuckovicSahutoglu09}.

\begin{proposition}[\cite{CuckovicSahutoglu09}]\label{prophull}
Let $\Omega\subset \mathbb{C}^n$ be a bounded convex domain.  Let $F:\overline{\mathbb{D}}\rightarrow b\Omega$ be a non-constant holomorphic map.  Then the convex hull of $F(\mathbb{D})$ is an affine analytic variety.
\end{proposition}

We note there are no analytic disks in the boundary of $B((p_1,p_2,...,p_n),r)$ because of convexity and the fact that Property (P) (see \cite{Cat}) is satisfied on the boundary.

We define the following directional derivatives.  We assume $\phi\in C(\overline{\Omega)}$.  Let $\vec{U}=(u_1,u_2)$ be a unit complex tangential vector at $p:=(p_1,p_2)\in b\Omega$.
Then if they exist as pointwise limits, 
\[\partial_b^{\vec{U},p}\phi:=\lim_{t\rightarrow 0}\frac{\phi(p_1+tu_1,p_2+tu_2)-\phi(p_1,p_2)}{t}\]
and
\[\overline{\partial}_b^{\vec{U},p}\phi:=\lim_{t\rightarrow 0}\frac{\phi(\overline{p_1+tu_1},\overline{p_2+tu_2})-\phi(\overline{p_1},\overline{p_2})}{t}\]
The following lemma uses these directional derivatives to characterize when a continuous function $\phi$ is holomorphic 'along' analytic disks in the boundary of the domain.  

\begin{lemma}\label{directional}
Let $\Omega\subset \mathbb{C}^2$ be a bounded pseudoconvex complete Reinhardt domain with a smooth boundary.  Suppose $\phi\in C(\overline{\Omega})$.  Then $\phi\circ g$ is holomorphic for any non-constant holomorphic $g:\mathbb{D}\rightarrow b\Omega$ if and only if for every $p\in g(\mathbb{D})$ and $\vec{U}$ tangent to $b\Omega$ at $p$,
\[\partial_b^{\vec{U},p}\phi\] exists as a pointwise limit and
\[\overline{\partial}_b^{\vec{U},p}\phi=0\]
\end{lemma}

\begin{proof}

Suppose $\phi\circ f$ is holomorphic for any $f:\mathbb{D}\rightarrow b\Omega$ holomorphic.  The first case to consider is if $\overline{f(\mathbb{D})}$ intersects either coordinate axis. 

If $\overline{f(\mathbb{D})}$ intersects either coordinate axis, then by smoothness of $b\Omega$, $f(\mathbb{D})$ is contained in an affine analytic variety and is either vertical or horizontal.  That is, $f(\mathbb{D})$ is contained in the biholomorphic image of $\mathbb{D}$.  And so one can show 
$\partial_b^{\vec{U},p}\phi$ exists and $\overline{\partial}_b^{\vec{U},p}\phi=0$.\\

That is, we may assume $f:=(f_1,f_2):\mathbb{D}\rightarrow b\Omega$ is holomorphic and neither $f_1$ nor $f_2$ is identically constant.  This implies $\overline{f(\mathbb{D})}$ is away from either coordinate axis.  Then $f(\mathbb{D})$ is contained in a family of analytic disks $\Gamma$ which foliate the boundary near $f(\mathbb{D})$.  Let $p\in f(\mathbb{D})$.  By Lemma \ref{almost} and Proposition \ref{prophull}, there exists a biholomorphism $T:B(p,r)\rightarrow \mathbb{C}^2$ so that 
$T(f(\mathbb{D}))\subset \mathbb{C}\times \{\alpha\}$ for some $\alpha\in [-1,1]$.  Furthermore, we may assume $T\circ f:=g$
where $g=(g_1,\alpha)$ and $g_1:\mathbb{D}\rightarrow \mathbb{C}$ is a biholomorphism with a continuous extension to the unit circle.  We may assume $g_1$ is a biholomophism by Proposition \ref{propconvex}.

Let $\phi\circ T^{-1}=\widetilde{\phi}$.  We will first show the tangential directional derivative $\partial_b^{\vec{U},p}\widetilde{\phi}$ and the conjugate tangential directional derivative $\overline{\partial}_b^{\vec{U},p}\widetilde{\phi}$ exists on $T(\Gamma)\subset \{(z_1,\alpha): z_1\in \mathbb{C}\land \alpha\in [-1,1]\}$ and $\overline{\partial}_b^{\vec{U},p}\widetilde{\phi}=0$ on $T(\Gamma)$ if and only if $\widetilde{\phi}\circ g$ is holomorphic for any holomorphic $g$ so that 
$g(\mathbb{D})\subset T(\Gamma)$.  First we suppose $\widetilde{\phi}\circ g$ is holomorphic and $g(\mathbb{D})\subset T(\Gamma)$.    Then we consider a unit vector $\vec{U}=(u,0)$ so that $\vec{U}$ is tangent to $g(\mathbb{D})$.  We may consider the restriction of $\phi$ to $\overline{T(\Gamma)}$ to be a function of $(z_1,\overline{z_1},\alpha)$.  That is,
\[\phi|_{\overline{T(\Gamma)}}=\phi(z_1,\overline{z_1},\alpha).\] Then for $({p_1},\alpha)\in g(\mathbb{D})$ we chose $t_0\in \mathbb{R}\setminus\{0\}$ so that for all $t$, $|t_0|>|t|>0$ we have $({p_1+tu},\alpha)\in g(\mathbb{D})$.  Then using the fact that $\widetilde{\phi}\circ g$ is holomorphic, we have

\begin{align*}
&\frac{\widetilde{\phi}(p_1,\overline{p_1+tu},\alpha)-\widetilde{\phi}(p_1,\overline{p_1},\alpha)}{t}\\
=&\frac{\widetilde{\phi}(g_1\circ g_1^{-1}(p_1),\overline{g_1}\circ \overline{g_1^{-1}}(\overline{p_1+tu}),\alpha)-\widetilde{\phi}(g_1\circ g_1^{-1}(p_1),\overline{g_1}\circ \overline{g_1^{-1}}(\overline{p_1}),\alpha)}{t}\\
\rightarrow &\frac{\partial(\phi\circ g\circ g_1^{-1})}{\partial \overline{z_1}}=0\\
\end{align*}
as $t\rightarrow 0$ and at $(p_1,\alpha)\in g(\mathbb{D})$.  By a similar argument, it can be shown that 
\[\partial_b^{\vec{U},p}\widetilde{\phi}:=\lim_{t\rightarrow 0}\frac{\widetilde{\phi}({p_1+tu},\alpha)-\widetilde{\phi}({p_1},\alpha)}{t}\] exists and is finite on $T(\Gamma)$.

Next we assume \[\overline{\partial}_b^{\vec{U},p}\widetilde{\phi}:=\lim_{t\rightarrow 0}\frac{\widetilde{\phi}(\overline{p_1+tu},\alpha)-\widetilde{\phi}(\overline{p_1},\alpha)}{t}=0\] on $T(\Gamma)$
and \[ \partial_b^{\vec{U},p}\widetilde{\phi}:=\lim_{t\rightarrow 0}\frac{\widetilde{\phi}({p_1+tu},\alpha)-\widetilde{\phi}({p_1},\alpha)}{t}\] exists and is finite on $T(\Gamma)$. 

Then 
\[\frac{\partial (\widetilde{\phi}\circ g)(\zeta)}{\partial \overline{\zeta}}=\partial_b^{\vec{U},p}\widetilde{\phi}\frac{\partial g}{\partial\overline{\zeta}}+\overline{\partial}_b^{\vec{U},p}\widetilde{\phi}\frac{\partial \overline{g}}{\partial\overline{\zeta}}=0\]

so by composing $\widetilde{\phi}$ with $T$, we have that $\phi\circ f$ is holomorphic.

\end{proof}

 \begin{proposition}\label{approx}
 Let $\Omega\subset \mathbb{C}^2$ be a bounded pseudoconvex complete Reinhardt domain with a smooth boundary.  Suppose $\phi\in C(\overline{\Omega})$ is such that $\phi\circ f$ is holomorphic for any holomorphic $f:\mathbb{D}\rightarrow b\Omega$.  Let $\Gamma\subset b\Omega$ be a continuous family of non-trivial analytic disks so that $\overline{\Gamma}$ is disjoint from the closure of any other non-trivial family of analytic disks in $b\Omega$.  Then there exists $\{\psi_n\}_{n\in \mathbb{N}}\subset C^{\infty}(\overline{\Omega})$ so that the following holds.
 \begin{enumerate}
 \item $\phi_n\rightarrow \phi$ uniformly on $\overline{\Gamma}$ as $n\rightarrow \infty$.
 \item $\phi_n\circ f$ is holomorphic for any holomorphic $f$ so that $f(\mathbb{D})\subset \Gamma$
 \end{enumerate}
 \end{proposition}
 
 \begin{proof}
Let $\nabla\subset b\Omega$ be a non-degenerate analytic disk so that $f(\mathbb{D})=\nabla$ where 
$f=(f_1,f_2)$ is holomorphic and continuous up to $\overline{\mathbb{D}}$.  Furthermore, assume $\nabla$ is away from the coordinate 
axes.  By Lemma \ref{lembiholo} and Proposition \ref{prophull}, there is a local holomorphic change of coordinates $T$ so that $T(\nabla)$ is contained in an affine analytic variety.  By Proposition \ref{propconvex}, we may assume $T(\nabla)$ is convex and $\overline{T(\nabla)}\subset \overline{T(\Gamma)}$ where $\Gamma$ is the continuous family of disks containing $f(\mathbb{D})$ and away from the closure of any other non-degenerate analytic disk.  Then the restriction $\phi|_{\Gamma}=\phi(z_1,\alpha)$ where $z_1\in (T(U\cap b\Omega))\subset\{(z_1,z_2)\in \mathbb{C}^2:z_2=0\}$ and $\alpha\in [-1,1]$.  Without loss of generality, extend $\phi$ as a continuous function on $\mathbb{C}^2$.  As notation, $\mathbb{D}_{\frac{1}{n}}:=\{z\in \mathbb{C}:|z|<\frac{1}{n}\}$.  

We let $\chi\in C^{\infty}_0(\mathbb{D})$ so that $0\leq \chi\leq 1$, $\chi$ is radially symmetric, and 
$\int_{\mathbb{C}}\chi=1$.\\

Similarly, we let $\widetilde{\chi}\in C^{\infty}_0(-1,1)$, $0\leq \widetilde{\chi}\leq 1$, and radially symmetric so that $\int_{\mathbb{R}}\widetilde{\chi}=1$.\\

Then we define the smooth mollifier $\{\chi_n\}_{n\in \mathbb{N}}\subset C^{\infty}_0(\mathbb{D}_{\frac{1}{n}}\times \left(-\frac{1}{n},\frac{1}{n}\right))$ as \[\chi_n(z_1,\alpha):=n^3\chi(nz_1)\widetilde{\chi}(n\alpha).\]

Then, there exists a holomorphic change of coordinates $H:V\rightarrow \mathbb{C}^2$ so that $T(\Gamma)\subset V$ and $H(T(\Gamma))=\mathbb{D}_s\times (-1,1)$ for some fixed radius $s>0$.  For every $n\in \mathbb{N}$, chose $0<r_n<1$ so that \[-1<r_n(\alpha-\beta)<1\] and \[|r_n(z_1-\lambda)|<s\] for every $(z_1,\alpha)\in \mathbb{D}_s\times (-1,1)$ and for all

 \[(\lambda,\beta)\in \mathbb{D}_{\frac{1}{n}}\times \left(-\frac{1}{n},\frac{1}{n}\right).\]

Then we define the convolution of $\phi\circ T^{-1}$ with $\{\chi_n\}$ in the following manner.
\[\psi_n(z_1,\alpha):=\int_{\mathbb{C}\times \mathbb{R}}\phi\circ T^{-1}(r_n(z_1-\lambda),r_n(\alpha-\beta))\chi_n(\lambda,\beta)dA(\lambda)d\beta.\]  

Let us extend $\psi_n$ trivially to $\mathbb{C}^2$ and denote this trivial extension as $\psi_n$, abusing the notation.

Now, we have everything we need to show $\psi_n\circ g$ are holomorphic for $g:\mathbb{D}\rightarrow T(\Gamma)$ holomorphic.

Using Lemma \ref{directional}, for every $n\in \mathbb{N}$,
\[\lim_{t\rightarrow 0}\frac{\phi\circ T^{-1}(r_n(\overline{z_1+tu}-\lambda),r_n(\alpha-\beta))-\phi\circ T^{-1}(r_n(\overline{z_1}-\lambda),r_n(\alpha-\beta))}{t}=0\] 

pointwise and

\[\lim_{t\rightarrow 0}\frac{\phi\circ T^{-1}(r_n({z_1+tu}-\lambda),r_n(\alpha-\beta))-\phi\circ T^{-1}(r_n({z_1}-\lambda),r_n(\alpha-\beta))}{t}\] exists and is finite for every $(\lambda,\beta)\in \mathbb{D}_{\frac{1}{n}}\times (-\frac{1}{n},\frac{1}{n})$.  Therefore, using the fact that $\chi_n$ are compactly supported and using the Lebesgue dominated convergence theorem, we have that 
\[\overline{\partial}_b^{\vec{U},p}\psi_n=0\] for any unit vector $\vec{U}$ tangent to $T(\Gamma)$ for all $p\in T(\Gamma)$, and for all $n\in \mathbb{N}$.  Furthermore, \[{\partial}_b^{\vec{U},p}\psi_n\] exists for any unit vector $\vec{U}$ tangent to $T(\Gamma)$, $p\in T(\Gamma)$, and $n\in \mathbb{N}$.  Therefore by Lemma \ref{directional}, $\psi_n$ are holomorphic along analytic disks in $T(\Gamma)$.

Furthermore, it can be shown that $\psi_n\circ T\rightarrow \phi$ uniformly on $\overline{\Gamma}$ as $n\rightarrow \infty$.  Now if $\Gamma$ intersects the coordinate axes, then the analytic disks are horizontal or vertical by smoothness of $b\Omega$.  So, we perform the convolution procedure as in \cite{ClosSahut} without using a holomorphic change of coordinates.

 \end{proof}
 
 For a linear operator $T:G\rightarrow H$ between Hilbert spaces, we define the essential norm
 as \[\|T\|_e:=\inf\{\|T-K\|\, ,K:G\rightarrow H\,\text{compact}\}\]

\begin{proposition}[\cite{ClosSahut}]\label{enorm}
Let $\Omega\subset \mathbb{C}^n$ be a bounded convex domain.  Suppose $\Gamma_{\Omega}\neq \emptyset$ is defined as above.
Assume $\{\phi_n\}_{n\in \mathbb{N}}\subset C(\overline{\Omega})$ so that $\phi_n\rightarrow 0$ uniformly on $\Gamma_{\Omega}$ as 
$n\rightarrow \infty$.  Then, $\lim_{n\rightarrow \infty}\|H_{\phi_n}\|_e=0$ 

\end{proposition}

The next proposition is similar to the theorem in \cite{CuckovicSahutoglu09}, with one major difference, namely they assumed smoothness of the boundary.  Here, we assume the boundary is piecewise
smooth.
 
 \begin{proposition}\label{onefamily}
 Let $\Omega\subset \mathbb{C}^2$ be a bounded convex domain so that the boundary of $\Omega$ contains no analytic disks except for one continuous family, called $\Gamma_{\Omega}$.  Let $\phi\in C^{\infty}(\overline{\Omega})$ so that $\phi\circ f$ is holomorphic for any holomorphic $f:\mathbb{D}\rightarrow b\Omega$.    
 Then, $H_{\phi}$ is compact on $A^2(\Omega)$.
 \end{proposition}

 \begin{proof}
 Without loss of generality, we may assume \[\Gamma_{\Omega}\subset \{(z_1,\alpha):z_1\in \mathbb{C}\, , \alpha\in (-1,1)\}.\] 
 Assuming $\phi\circ f$ is holomorphic for any $f:\mathbb{D}\rightarrow b\Omega$, one can show that the tangential directional derivative $\overline{\partial}_b \phi$ exists along $\Gamma_{\Omega}$.  Furthermore $\frac{\partial\phi}{\partial\overline{z_1}}=0$ on $\Gamma_{\Omega}$.  We wish to construct smooth function $\psi\in C^{\infty}(\overline{\Omega})$ so that $\psi \equiv \phi$ on $\Gamma_{\Omega}$ and $\overline{\partial}(\psi)=0$ on $\Gamma_{\Omega}$.  To do this, we will use the idea of a defining function.
There exists a smooth function $\rho\in C^{\infty}(\mathbb{C}^2)$ so that $\rho\equiv 0$ on $\overline{\{(z_1,\alpha):z_1\in \mathbb{C}\, , \alpha\in (-1,1)\}}$ and $|\nabla \rho|>0$ on $\overline{\{(z_1,\alpha):z_1\in \mathbb{C}\, , \alpha\in (-1,1)\}}$.  Furthermore, by scaling the tangential and normal vector fields on $\overline{\{(z_1,\alpha):z_1\in \mathbb{C}\, , \alpha\in (-1,1)\}}$, we may assume 
\[\frac{\partial\rho}{\partial\overline{z_1}}|_{\overline{\{(z_1,\alpha):z_1\in \mathbb{C}\, , \alpha\in (-1,1)\}}}=0\] and \[\frac{\partial\rho}{\partial\overline{z_2}}|_{\overline{\{(z_1,\alpha):z_1\in \mathbb{C}\, , \alpha\in (-1,1)\}}}=1.\]  Now we define 
\[\psi:=\phi-\rho\left(\frac{\partial\phi}{\partial\overline{z_2}}\right).\]
Then $\overline{\partial}\psi=0$ on $\Gamma_{\Omega}$ and also $\psi=\phi$ on $\Gamma_{\Omega}$.  Then by
Proposition \ref{enorm}, $\|H_{\phi-\psi}\|_e=0$ and so $H_{\phi-\psi}$ is compact on $A^2(\Omega)$.  To show $H_{\psi}$ is compact we use the fact that $\overline{\partial}\psi=0$ on $\Gamma_{\Omega}$ together with the same argument seen in \cite{CuckovicSahutoglu09} that shows $H_{\widetilde{\beta}}$ is compact if $\overline{\partial}\widetilde{\beta}=0$ on $\Gamma_{\Omega}$.  Therefore we conclude $H_{\phi}$ is compact. 

 \end{proof}
 
 
 
\section{Proof of Theorem \ref{thmmain}}

The idea is to use the following result which will allow us to localize the 
problem.  
\begin{proposition}[\cite{CuckovicSahutoglu09}]\label{local}
Let $\Omega\subset \mathbb{C}^n$ for $n\geq 2$ be a bounded pseudoconvex domain and $\phi\in L^{\infty}(\Omega)$.  If for every $p\in b\Omega$ there exists an open neighbourhood $U$ of $p$ such that $U\cap \Omega$ is a domain and 
\[H^{U\cap \Omega}_{R_{U\cap\Omega}(\phi)}R_{U\cap\Omega}\] is compact on $A^2(\Omega)$, then $H^{\Omega}_{\phi}$ is compact on $A^2(\Omega)$.

\end{proposition} 

We will also use the following lemma appearing in \cite{CuckovicSahutoglu09}.

\begin{lemma}[\cite{CuckovicSahutoglu09}]\label{bi}
Let $\Omega_1$ and $\Omega_2$ be bounded pseudoconvex subsets of $\mathbb{C}^n$.  Suppose 
$\phi\in C^{\infty}(\overline{\Omega_1})$ so that $H_{\phi}$ is compact on $A^2(\Omega_1)$.  Let $T:\Omega_2\rightarrow \Omega_1$ be a biholomorphism with a smooth extension to the boundary.  Then $H_{\phi\circ T}$ is compact on $A^2( \Omega_1)$.

\end{lemma}

As we shall see, this collection of all non-constant analytic disks in $b\Omega$ will play a crucial role in our understanding of the compactness of Hankel operators on various domains in $\mathbb{C}^n$ for $n\geq 2$.  There are several cases to consider depending on where $p\in b\Omega$ is located. 

\begin{enumerate}
\item $p\in \Gamma_{\Omega}\subset b\Omega$ but away from the coordinate axes.
\item $p\in b\Omega\setminus \Gamma_{\Omega}$.
\item $p\in \{z_1=0\}\cup \{z_2=0\}$
\end{enumerate}

\vspace{.1in}

We will first consider the case where $p$ is away from $\Gamma_{\Omega}$. We let $p:=(p_1,p_2)\in b\Omega$ and assume $p\in b\Omega\setminus \Gamma_{\Omega}$.  So there exists an $r>0$ sufficiently small so that the ball $b(B(p,r)\cap \Omega)$ contains no analytic disks.  Furthermore, there exists a biholomorphism $T:B(p,r)\rightarrow \mathbb{C}^2$ so that $T(B(p,r)\cap \Omega)$ is a convex domain. Therefore, since any analytic disk in $bT(B(p,r)\cap \Omega)$ must be the image (under $T$) of a disk in $b(B(p,r)\cap\Omega)$, there are no analytic disks in $bT(B(p,r)\cap \Omega)$.  By convexity and compactness of the $\overline{\partial}$-Neumann operator, the Hankel operator 
\[H_{\phi\circ T^{-1}}^{T(B(p,r)\cap \Omega)}\] is compact on $A^2(T(B(p,r)\cap \Omega))$.  And so 
this proves $H^{U\cap \Omega}_{R_{U\cap\Omega}(\phi)}$ is compact on $A^2(U\cap\Omega)$ where $U:=B(p,r)$.\\

If $p\in (\{z_1=0\}\cup \{z_2=0\})\cap b\Omega$, then by smoothness of the domain, either $p$ is contained in an analytic disk, $p$ is a limit point of a sequence of analytic disks, or $p$ is contained in part of the boundary satisfying property (P).  If $p\in b\Omega$ is contained in a non-degenerate analytic disk, then locally the analytic disks are horizontal or vertical, by smoothness of the domain.  Without loss of generality, assume the family of analytic disk is vertical.  So, using the argument in \cite{ClosSahut}, we can approximate the continuous symbol $\phi$ uniformly on $\Gamma_{U\cap\Omega}$ for some ball $U$ centered at $p$ with a sequence of smooth functions $\psi_n$ so that $\psi_n$ is holomorphic along any analytic disk contained in $b(U\cap\Omega)$.  As in \cite{ClosSahut}, we use \cite{CuckovicSahutoglu09} and the uniform approximation on $\Gamma_{U\cap\Omega}$ to conclude that $H^{U\cap\Omega}_{\phi|_{U\cap\Omega}}$ is compact on $A^2(U\cap\Omega)$.\\

Note that if $p\in b\Omega$ is contained in part of the boundary satisfying property (P) (see \cite{Cat}), then the local $\overline{\partial}$-Neumann operator $N_1^{U\cap\Omega}$ is compact since there exists a convex neighbourhood $U$ of $p$ so that $U\cap\Omega$ is convex, and so $H^{U\cap\Omega}_{\phi|_{U\cap\Omega}}$ is compact on $A^2(U\cap\Omega)$.\\

Lastly, if $p\in b\Omega\setminus (\{z_2=0\}\cup \{z_1=0\})$ and $p\in \Gamma_{\Omega}$.  We will first assume $p$ is contained in a limit set of a discrete sequence of families of analytic disks.  We may assume discreteness due to Lemma \ref{lembiholo}, Proposition \ref{propconvex}, and smoothness of the boundary of $\Omega$.  Then by Lemma \ref{disklim}, this limit set exactly equals $\{p\}$.  We will first assume $p$ is not contained in the closure of a single non-trivial analytic disk.

 Let $U:=B(p,r)$ chosen so that $U\cap \Omega$ is a domain and $T(U\cap\Omega)$ is convex for some biholomorphism $T:U\rightarrow \mathbb{C}^2$.  Denote this discrete collection of continuous families of analytic disks as $\{\Gamma_j\}_{j\in \mathbb{N}}\subset b(U\cap\Omega)$.  Furthermore, we may assume \[\Gamma_{T(U\cap\Omega)}=\bigcup_{j\in \mathbb{N}}\Gamma_j.\]

Then $\{T(\Gamma_j)\}_{j\in \mathbb{N}}$ is a discrete collection of families of affine analytic disks.  Then for each $j\in \mathbb{N}$ there exists open pairwise disjoint neighborhoods $V_j$ with a strongly pseudoconvex boundary so that $T(\Gamma_j)\subset V_j$.  Let $\rho_j$ be smooth cutoff functions so that $\rho_j\equiv 1$ on a neighborhood of $T(\Gamma_j)$ and $\rho_j$ are compactly supported in $V_j$.  Define \[\widetilde{\phi}_j:=\rho_j (\phi\circ T^{-1}-\phi(p_1,p_2)).\]  We wish to show $H_{ \widetilde{\phi}_j}$ are compact on $A^2(T(U\cap\Omega))$ for all $j\in \mathbb{N}$.  By Lemma \ref{onefamily} and Proposition \ref{approx}, we approximate $\phi\circ T^{-1}-\phi(p_1,p_2)$ with a sequence $\{\psi^j_n\}_{n\in \mathbb{N}}\subset C^{\infty}(\mathbb{C}^2)$ so that 
$\psi^j_n\rightarrow \phi\circ T^{-1}-\phi(p_1,p_2)$ uniformly on $\overline{T(\Gamma_j)}$ as $n\rightarrow \infty$ and 
$\psi^j_n$ are holomorphic along $T(\Gamma_j)$.  Then, $\rho_j\psi^j_n$ are holomorphic along any analytic disk in $bT(U\cap\Omega)$ for all $j,n\in \mathbb{N}$ and $\rho_j\psi^j_n\in C^{\infty}(\mathbb{C}^2)$.  Fix $j, n\in \mathbb{N}$.  Then, there exists a function 
$\delta_{j,n}\in C^{\infty}(\mathbb{C}^2)$ so that 
\begin{enumerate}
\item $\overline{\partial}\delta_{j,n}=0$ on $\Gamma_{T(U\cap\Omega)}$.
\item $\delta_{j,n}=\rho_j\psi^j_n$ on $\Gamma_{T(U\cap\Omega)}$.

\end{enumerate}

Therefore by an argument similar to the proof of Proposition \ref{onefamily}, $H^{T(U\cap\Omega)}_{\delta_{j,n}}$,
$H^{T(U\cap\Omega)}_{\rho_j\psi^j_n-\delta_{j,n}}$, and therefore $H^{T(U\cap\Omega)}_{\rho_j\psi^j_n}$ are compact on $A^2(T(U\cap\Omega))$ for all $j,n\in \mathbb{N}$.\\

Furthermore, \[\rho_j\psi^j_n\rightarrow \widetilde{\phi}_j\] uniformly on $\Gamma_{T(U\cap\Omega)}$ as $n\rightarrow \infty$. Then by convexity of $T(U\cap\Omega)$ and Proposition \ref{enorm}, $H_{\widetilde{\phi}_j}$ are compact on $A^2(T(U\cap\Omega))$ for all $j\in \mathbb{N}$.  One can show that 
\[\alpha_N:=\sum_{j=1}^N\widetilde{\phi}_j\] converges uniformly to $\phi\circ T^{-1}-\phi(p_1,p_2)$ on $\Gamma_{T(U\cap\Omega)}$ as $N\rightarrow \infty$.  Also, $H_{\alpha_N}$ are compact on $A^2(T(U\cap\Omega))$ for all $N\in \mathbb{N}$ as the finite sum of compact operators.  Furthermore, $\alpha_N\in C^{\infty}(\overline{T(U\cap\Omega)})$ for all $N$.  Then by Lemma \ref{bi}, 
$H_{\alpha_N\circ T}$ are compact on $A^2(U\cap\Omega)$ for all $N$ and so 
$H^{U\cap\Omega}_{\phi|_{U\cap\Omega}}$ is compact on $A^2(U\cap \Omega)$.

So, we have the following.  For all $p:=(p_1,p_2)\in b\Omega$ there exists $r>0$ so that 
$B(p,r)\cap\Omega$ is a domain and \[H^{B(p,r)\cap\Omega}_{\phi|_{U\cap\Omega}}\] is compact on $A^2(B(p,r)\cap \Omega)$.
Then by composing with the restriction operator $R:A^2(\Omega)\rightarrow A^2(B(p,r)\cap\Omega)$, we have
that \[H^{B(p,r)\cap\Omega}_{\phi|_{U\cap\Omega}}R\] is compact on $A^2(\Omega)$.  Then by Proposition \ref{local},
$H_{\phi}$ is compact on $A^2(\Omega)$.

Next, we assume there exists a non-trivial analytic disk $\Gamma_0\in bT(U\cap\Omega)$ so that $p\in \overline{\Gamma_0}$ and 
$\{p\}$ is the limit set of $\{\Gamma_j\}_{j\geq 1}$.  Then we can represent 
\[\Gamma_{U\cap\Omega}=\bigcup_{j\geq 0,\,\theta\in [0,2\pi]}\{e^{i\theta}\Gamma_j\}.\]
For $0<r<1$ we define 
\[\Gamma_r:=\bigcup_{f(\mathbb{D})\subset \Gamma_{U\cap\Omega},\,\theta\in [0,2\pi]}\{e^{i\theta}f(r\mathbb{D})\}\]
By convolving $\phi$ with a mollifier in $[0,2\pi]$, there exists $\{\tau_n\}_{n\in \mathbb{N}}\subset C(\overline{\Omega})$ so
that $\tau_n\rightarrow \phi$ uniformly on $\overline{\Gamma_r}$ as $n\rightarrow \infty$, and 
for every $(z_1,z_2)\in \Gamma_r$ and $\vec{T}$ complex tangent to $bU\cap\Omega$ at $(z_1,z_2)$ the directional derivative 
of $\tau_n$ in the direction of $\vec{T}$ at $(z_1,z_2)$ exists.  Furthermore, by the smoothness of $\tau_n$ in the $\theta$ variable, 
the directional derivative in the complex normal direction at $(z_1,z_2)$ also exists.  Thus $\tau_n$ satisfies the compatibility condition
for the Whitney extension theorem.  See \cite{stein} and \cite{mal} for more information on the Whitney extension theorem.  Therefore, there exits $\widetilde{\tau}_n\in C^1(\overline{\Omega})$ so that
$\widetilde{\tau}_n\equiv \tau_n$ on $\Gamma_r$ and both tangential and normal directional derivatives of $\widetilde{\tau}_n$ agree with $\tau_n$.  That is, $\tau_n\circ f$ are holomorphic on $\mathbb{D}$ for any $n\in \mathbb{N}$ and $f(\mathbb{D})\subset \Gamma_{U\cap\Omega}$.  Thus $H_{\tau_n}$ is compact on $A^2(\Omega)$ by \cite{CuckovicSahutoglu09} and Proposition \ref{enorm}.  And so using Proposition \ref{enorm} again and letting $r\rightarrow 1^-$, we conclude $H_{\phi|_{U\cap\Omega}}R_{U\cap\Omega}$ is compact on $A^2(\Omega)$.  And so by Proposition \ref{local}, $H_{\phi}$ is compact on $A^2(\Omega)$.

 \bibliographystyle{amsalpha}
\bibliography{rrrefs}

 \end{document}